\documentclass[reqno,12pt,letterpaper,proof]{amsart}
\usepackage{amsmath,amsthm,amssymb}
\usepackage{mathtools}
\usepackage{geometry}
\usepackage{enumitem}

\RequirePackage{amsthm,graphicx,mathrsfs,url, bbm}
\RequirePackage[usenames,dvipsnames]{color}
\RequirePackage[colorlinks=true,linkcolor=ForestGreen,citecolor=NavyBlue]{hyperref}
\RequirePackage{amsxtra}
\usepackage{comment}
\usepackage{cases}

\newtheorem{theorem}{Theorem}
\newtheorem{proposition}{Proposition}[section]
\newtheorem{lemma}[proposition]{Lemma}

\theoremstyle{definition}
\newtheorem{definition}[proposition]{Definition}

\def\Remark{\noindent\textbf{Remark.}\ }
\def\Remarks{\noindent\textbf{Remarks.}\ }

\newcommand{\C}{\mathbb{C}} 
\newcommand{\R}{\mathbb{R}} 
\newcommand{\N}{\mathbb{N}} 
\renewcommand{\H}{\mathbb{H}}
\newcommand{\W}{\mathbb{W}}
\renewcommand{\d}{\mathrm{d}}
\newcommand{\ztr}{\textup{0-tr}}

\relax

\let\oldtocsection=\tocsection
\let\oldtocsubsection=\tocsubsection
\renewcommand{\tocsection}[2]{\hspace{0em}\oldtocsection{#1}{#2}}
\renewcommand{\tocsubsection}[2]{\hspace{1em}\oldtocsubsection{#1}{#2}}

\numberwithin{equation}{section}

\title{Brownian Loops and the Selberg Zeta Function}
\author{Roman Lemonde}
\address{Eidgenössische Technische Hochschule Zürich, ETHZ}
\email[Roman Lemonde]{rlemonde@ethz.ch}
\author{Jian Wang}
\address{Institut des Hautes {\'E}tudes Scientifiques, 91893 Bures-sur-Yvette, France}
\email[Jian Wang]{wangjian@ihes.fr}

\begin{document}

\begin{abstract}

We study the Brownian loop measure on hyperbolic surfaces for Brownian motion with a constant killing rate. 
We compute the mass of Brownian loops with killing in a free homotopy class and then relate the total mass of loops in all essential homotopy classes to the Selberg zeta function when the surface is geometrically finite. As an application, we provide a probabilistic interpretation of different notions of regularized determinants of Laplacian, in both the compact and infinite-area cases.
\end{abstract}

\maketitle

\thispagestyle{empty}

\section{Introduction}
Let $\Gamma\subset \mathrm{PSL}_2(\R)$ be a torsion-free Fuchsian group and $X\coloneqq \Gamma\backslash \H$ be a complete, geometrically-finite hyperbolic surface without boundary. Let $\mathcal P_X$ be the set of {\em oriented} primitive closed geodesics. The Selberg zeta function was first introduced by Atle Selberg as a tool to study the Selberg trace formula (see \cite{fischer}). We define it as
\[\label{def selberg zeta function} Z_X(s)\coloneqq \prod_{\gamma\in \mathcal P_X}\prod_{k=0}^\infty \left( 1-e^{-(s+k)\ell_\gamma} \right)\]
which is convergent when \(\mathrm{Re}(s)> \delta\) (see Lemma \ref{convergence zeta function}). Here, $\ell_\gamma$ is the hyperbolic length of $\gamma\in \mathcal P_X$, and $\delta$ is the exponent of convergence of $\Gamma$, defined as the infimum of $\delta'>0$ such that the Poincar\'e series $\sum_{h\in \Gamma} e^{-\delta' d(z,h\cdot z)}$ is finite for all $z\in \H$.

The Brownian loop measure was first introduced on the plane $\C$ in \cite{lawler} and \cite{lawler_conf_rest}. Its definition can be generalized to any orientable Riemannian surface $(X,g)$, and to Brownian motion with a constant killing rate $\kappa$. We denote it $\mu_X^\kappa$. It is defined on the space of oriented loops on $X$ (see \S \ref{def_loop_measure} for a brief introduction). The main 
result of this paper 
relates the total mass of the Brownian loop measure 
in 
all essential homotopy classes to the Selberg zeta function:
\begin{theorem}\label{thm:zeta} 
    Let $X$ be as above. Then for any $\kappa\geq-\frac14$ such that $\frac12+\sqrt{\frac14+\kappa}>\delta$,
    \[ \sum_{\gamma\in \mathcal P_X}\sum_{m=1}^\infty \mu^\kappa_X\left( \mathcal C_X(\gamma^m) \right)=-\log Z_X\left( \frac12+\sqrt{\frac14+\kappa} \right), \]
where for each $\gamma\in \mathcal P_X$, $m\in \mathbb N^*$, $\mathcal C_X(\gamma^m)$ denotes the set of all oriented closed curves on $X$ that are freely homotopic to $\gamma^m$.
\end{theorem}
This theorem provides a probabilistic interpretation of the value of the Selberg zeta function on the real interval $(\delta, \infty)$ in terms of the Brownian loop measure with killing.
We discuss some immediate observations of Theorem \ref{thm:zeta} in the remarks below.

\begin{figure}[t]
    \centering
    \includegraphics[width=.9\textwidth]{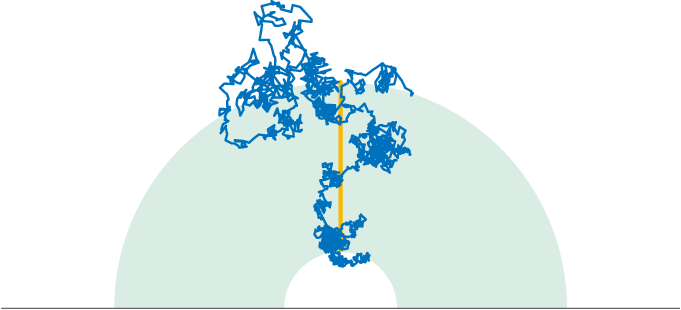}
    \caption{Illustration of a closed geodesic (yellow) and a Brownian loop (blue) of the same homotopic type on the hyperbolic cylinder. Both the geodesic and the Brownian loop are lifted to the universal cover $\H$. The light green region is a fundamental domain of the hyperbolic cylinder.}
    \label{fig:loop}
\end{figure}

\Remarks
1. Since $Z_X(s)$ converges for $s>\delta$, an immediate consequence of Theorem \ref{thm:zeta} is that the total mass of the Brownian loop measure in all essential homotopy classes is {\em finite} when the killing rate satisfies $\frac12+\sqrt{\frac14+\kappa}>\delta$. This generalizes \cite[Corollary 4.9]{BLM} 
where the case $\kappa=0$ is considered. In fact, since $\delta<1$ for $X$ geometrically finite with infinite area (see, for instance, \cite[\S 2.5.2]{Borth}), Theorem~\ref{thm:zeta} shows that the Brownian loop measure without killing ($\kappa=0$) of all loops in essential homotopy classes is finite. It is also worth noting that homotopy questions around the Brownian loop measure with a killing rate were studied in \cite{lejan}.

\noindent
2. Since the eigenvalues and resonances of the Laplacian on $X$ determine $Z_X$, as shown in \cite[Theorem 3.1]{borth_judge_perry}, Theorem \ref{thm:zeta} implies that they also determine the total mass of the Brownian loop measure with killing for homotopically essential loops. A more precise relation is stated in Proposition \ref{proposition:zeta_factor}.

\noindent
3. Note that the Selberg zeta function is not always well defined for geometrically infinite surfaces, as the product would diverge for any value of $s$. In this sense, for the sake of simplicity, all surfaces considered in this paper will be geometrically finite.

\bigskip

An important motivation for studying the total mass of the Brownian loop measure is its relation to the renormalized integral of the heat trace and regularized determinants of the Laplacian. In the compact case, with or without boundary, it was remarked in \cite{dubedat, lejan_markov_loops} that the logarithm of the zeta-regularized determinant  of the Laplacian $-\log\mathrm{det}_\zeta\Delta$ could be interpreted as the total mass of the Brownian loop measure. The following theorem provides a new renormalization procedure that elucidates this link. 

\begin{theorem}\label{thm:det}
    Suppose $X$ and $\kappa$ are as in Theorem \ref{thm:zeta}. 
    \begin{enumerate}[label=(\roman*)]
        \item Let $E_{X,\kappa}(t)$ be the renormalized heat operator as in Definition \ref{definition:trace}. Then 
        \[ \sum_{\gamma\in\mathcal{P}_X}\sum_{m=1}^\infty \mu_X^\kappa\big(\mathcal{C}_X(\gamma^m)\big) = \int_0^\infty \frac{1}{t}\operatorname{Tr}E_{X,\kappa}(t)\,\mathrm{d}t. \]
        \label{item:normlized_trace}

        \item Suppose $X$ is compact. Then, as $\kappa \to 0^+$,
    \begin{equation*}
        \sum_{\gamma\in \mathcal P_X}\sum_{m=1}^{\infty}\mu_X^\kappa(\mathcal C_X(\gamma^m))=-\log\mathrm{det}_\zeta\Delta+\mathrm{Area}(X)E-2\gamma_{\textup{EM}} -\log\kappa +o(1),
    \end{equation*}
    where  $\gamma_{\mathrm{EM}}$ is the Euler--Mascheroni constant, $\zeta$ is the Riemann zeta function, and $E=(4\zeta'(-1)-1/2+\log(2\pi))/(4\pi)$. \label{item:compact_expansion}
    
        \item Suppose $X$ has infinite area. Let $s=\frac 12 +\sqrt{\frac 14 +\kappa}$, and $\det_0(\Delta_X+s(s-1))$ be the regularized determinant of $\Delta_X$ as in \cite[\S 2]{borth_judge_perry} (called $D(s)$ there; see also \S \ref{section:infinite}). Then: 
        \[\begin{split} \sum_{\gamma\in\mathcal{P}_X}\sum_{m=1}^\infty\mu_X^\kappa(\mathcal{C}_X(\gamma^m)) 
        = & -\log {\det}_0(\Delta_X+s(s-1)) +Gs(1-s)+F\\
         & -n_c\log G_0(s)+\log\chi(X)\log G_{\infty}(s), 
        \end{split}\]
    where $F$ and $G$ are the two constants of integration left unspecified in the definition of $\det_0(\Delta_X+s(s-1))$, $n_c$ is the number of cusps of our surface, $\chi(X)$ is the Euler characteristic of $X$, $G_0(s)=2^{-s}(s-\frac12)^{\frac12}\Gamma(s-\frac12)$, and $G_\infty(s)=(2\pi)^{-s}\Gamma(s)G(s)^2$, where $G$ is the Barnes $G$-function (see \cite[p.56]{Borth}). \label{item:noncompact_expansion}
    \end{enumerate}
\end{theorem}

\Remarks
1. In view of Theorem \ref{thm:zeta}, the asymptotic expansion in item~\ref{item:compact_expansion} of Theorem~\ref{thm:det} is a consequence of an identity established by Sarnak \cite[Theorem 1]{sarnak}. A direct proof of the asymptotic expansion is provided in \S \ref{compact case of asymptotic expansion}. 

\noindent
2. In the context of closed smooth Riemannian surfaces $(M,g)$, the mass of loops on $M$ under $\mu_M^\kappa$ with a total quadratic variation of at least $4\varpi>0$ is defined as
\[ \mu_M^\kappa(\text{loops with quadratic variation}\geq 4 \varpi)\coloneqq \int_{\varpi}^\infty \mathrm{Tr}(e^{-t(\Delta_M+\kappa)}) \,\frac{\mathrm d t}{t}. \]
For a compact smooth Riemannian surface $(M,g)$, M. Ang et al. \cite{AngPark} showed that the mass of loops with quadratic variation of at least $4\varpi>0$, as defined above, admits the asymptotic expansion
\[ \frac{\mathrm{vol}_g(M)}{4\pi \varpi}
 - \frac{\chi(M)}{6}
 \left(\log \varpi + \gamma_{\textup{EM}}\right)
 - \log \kappa
 - \log {\det}_{\zeta} \Delta_M
 + \epsilon(\kappa,\varpi), \]
where ${\det}_{\zeta}\Delta_M$ is the zeta-regularized determinant of the Laplacian (see \S \ref{compact case of asymptotic expansion}), and $\epsilon(\kappa,\varpi)$ tends to zero in the limit as first $\kappa$ and then $\varpi$ are sent to zero. 
In the case where $M$ is hyperbolic, item~\ref{item:compact_expansion} of Theorem \ref{thm:det} is an analog of this asymptotic expansion, where no cut-off of loops with small quadratic variation is needed. 

\noindent
3. It is worth noting that a renormalization of the Laplacian determinant according to homotopy classes of the Brownian loop measure has already been carried out in the case without killing; see \cite[Proposition 4.11]{BLM}.

\noindent
4. For hyperbolic surfaces $X$ with infinite area, the Laplacian has a continuous spectrum; hence, its determinant cannot be renormalized by the zeta function regularization approach. Instead, D. Borthwick, C. Judge and P. A. Perry \cite[\S 2]{borth_judge_perry} defined the regularized determinant of $\Delta_X$ using the ``0-trace'' of the resolvent of the Laplacian (see \S \ref{section:infinite}). With such regularization, they are able to establish an identity connecting $\det_0(\Delta_X+s(s-1))$ (called $D(s)$ in \cite{borth_judge_perry}) and the Selberg zeta function; see \cite[Theorem 3.2]{borth_judge_perry}. Item~\ref{item:noncompact_expansion} of Theorem \ref{thm:det} provides a probabilistic interpretation to this result.  

\vspace{8pt}
\noindent{\bf Structure of the paper.}
In ~\S \ref{def_loop_measure}, we present the Brownian loop measure and its main properties. ~\S \ref{section Selberg zeta function} introduces the Selberg zeta function and highlights its spectral properties. In ~\S \ref{blm}, we consider the loop measure for Brownian motion with a constant killing rate, and generalize \cite[Lemma 3.2]{BLM} to the case with killing. This leads to 
the proofs of Theorems~\ref{thm:zeta}. In ~\S \ref{results}, we discuss renormalizations of the heat trace and determinants of Laplacians, and connect them with the mass of the Brownian loop measure, proving Theorem \ref{thm:det}.

\vspace{8pt}
\noindent{\bf Acknowledgments.}
The authors would like to thank Yilin Wang for introducing them to the subject and for many helpful discussions. We also thank Maciej Zworski for his interest in this work, and  
Giacomo Spriano for a number of productive conversations. JW is supported by Simons Foundation through a postdoctoral position at the Institut des Hautes Études Scientifiques. This work has received funding from the Swiss State Secretariat for Education, Research and Innovation (SERI): MB25.00004.

\section{Preliminaries}\label{preliminaries}

In this section, we review the necessary facts surrounding the Brownian loop measure and the Selberg zeta function of hyperbolic surfaces.

\subsection{Brownian loop measure.}\label{def_loop_measure}

Let $(X,g)$ be a Riemannian manifold without boundary of dimension $n$. We define Brownian motion (ran at speed 2) $W$ on $X$ as the diffusion process generated by $-\Delta_g$, where $\Delta_g$ is the (positive) Laplace--Beltrami operator
$$\Delta_g=-\text{div}_g\nabla_g=-\frac{1}{\sqrt{|g|}}\sum_{i,j=1}^n\partial_i(\sqrt{|g|}g^{ij}\partial_j).$$
We refer to \cite{lawler2005} and \cite{diffusion_processes} for the basics of Brownian motion.
This definition naturally generalizes that of Brownian motion in the Euclidean space $\mathbb{R}^n$. 
For $t\geq 0$, let $C([0,t];X)$ be the space of continuous functions equipped with the distance function 
\[ d_X^{\infty}(W, W')\coloneqq \sup_{s\in [0,t]} |W_s-W'_s|_X, \textup{ for } W, \ W'\in C([0,t];X). \]
For $x\in X$, let $\W_x^t(X)$ 
be the probability measure of the path of Brownian motion (ran at speed 2) starting at $x$ and running until time $t$ in the space $C([0,t];X)$.
In other words, for any measurable subset $A\subset C([0,t];X)$, $\mathbb W_x^t(X)(A)$ is the probability that the path of Brownian motion starting from $x$ until time $t$ is in $A$.
Notice that the total quadratic variation of a path under $\W_x^t(X)$ is almost surely $4t$. We can then decompose $\W_x^t(X)$ according to the endpoint 
$y\in X$ 
\begin{equation*}
    \W_{x\to y}^t(X)\coloneqq \lim_{\varepsilon\to 0^+}\frac{\W_x^t(X)\mathbbm{1}_{\{W_0=x, \ W_t\in D_{\varepsilon, X}(y)\}}}{\text{vol}_g(D_{\varepsilon, X}(y))}
\end{equation*}
where $\mathbbm 1_K$ denotes the restriction to $K\subset C([0,t];X)$ such that $\W_x^t(X)\mathbbm 1_K(A)\coloneqq \W_x^t(X)(K\cap A)$ for measurable $A\subset C([0,t];X)$, and $D_{\varepsilon,X}(y)$ is the geodesic ball of radius $\varepsilon$ in $X$ around $y$. For detailed justification and interpretation of $\mathbb W_{x\to y}^t(X)$, one can see \cite[\S 5.2]{lawler2005}.
\begin{lemma}\label{lemma:disintegrate}
    Suppose $t\geq 0$, $x,y\in X$. Let $\mathbb W_{x\to y}^t(X)$ be as above. Then 
    \[ \mathbb W_{x}^t(X)=\int_X \mathbb W_{x\to y}^t(X) \,\mathrm{d vol}_g(y). \]
    Moreover, the total mass of $\mathbb W_{x\to y}^t(X)$ on $C([0,t];X)$, denoted $|\mathbb W_{x\to y}^t(X)|$ below, is the integral kernel $p_X(t,x,y)$ of the heat operator.
\end{lemma}

\begin{proof}
For the disintegration formula, see \cite[\S 5.2]{lawler2005}. 
For the total mass of $\mathbb W_{x\to y}^t(X)$, notice that 
\[ |\mathbb W^t_{x\to y}(X)| = \lim_{\epsilon \to 0+}\frac{\mathbb P_x(W_t\in D_{\epsilon,X}(y))}{\mathrm{vol}_g(D_{\epsilon,X}(y))} = \lim_{\epsilon\to 0+}\frac{1}{\mathrm{vol}_g(D_{\epsilon,X}(y))}\int_{D_{\epsilon,X}(y)}p(t,x,y)\,\mathrm{dvol}_g(y) \]
where $\mathbb P_x$ is the measure of Brownian motion starting from $x$. The result then follows from the Lebesgue differentiation theorem.
\end{proof}

Let $\mathcal C_X^*$ be the set of parametrized loops:
\[ \mathcal C_X^*\coloneqq \{ W_\bullet \in C([0,t]; X) \mid t\geq 0, \ W_0 = W_t \}. \]
The \textit{parametrized Brownian loop measure} on $X$ is the measure on $\mathcal C_X^*$ defined as 
$$\mu_X^{*}\coloneqq \int_0^\infty \frac{\d t}{t}\int_X \W_{x\to x}^t(X) \,\text{dvol}_g(x).$$
We now define an equivalence relation on the space of parametrized loops $\mathcal C_X^*$. 
Let $W\in C([0,t]; X)$ such that $W_0=W_t$, and $W'\in C([0,t'];X)$ such that $W'_0=W'_{t'}$ be two parametrized loops in $\mathcal C_X^*$. We say $W\sim W'$ if and only if there exists an orientation preserving homeomorphism $\sigma$ from $[0,t]/(0\sim t)$ to $[0, t']/(0\sim t')$ such that $W'_{\sigma(s)}=W_s$ for all $s\in [0,t]$. The quotient space $\mathcal C_X\coloneqq \mathcal C_X^*/\sim$ consists of unrooted and unparametrized --- but oriented --- loops. The metric on $\mathcal C_X$ is defined by 
\[ d_X^{\textup{loop}}([W],[W'])\coloneqq \inf_{\sigma}\sup_{s}|W'_{\sigma(s)}-W_{s}|_X \]
where $\sigma$ runs through all orientation preserving homeomorphisms from $[0,t]/(0\sim t)$ to $[0,t']/(0\sim t')$, $s$ goes through $[0,t]$, and $|\bullet|_X$ is the geodesic distance on $X$.
The \textit{Brownian loop measure} $\mu_X$ is the measure on $\mathcal L$ induced by $\mu_X^*$ on $\mathcal C_X^*$, in the sense that $\mu_X(A) = \mu_X^*(\phi^{-1}(A))$ where $\phi: \mathcal C_X^*\to \mathcal C_X$ is the quotient map and $A$ is any measurable set in $\mathcal C_X$.  

For an oriented closed primitive geodesic $\gamma\in \mathcal P_X$, we denote $\mathcal C_X(\gamma)$ as the subset of $\mathcal C_X$ that consists of unparametrized, unrooted loops that are freely homotopic to $\gamma$.

The Brownian loop measure is Borel with respect to the topology induced by $d_X^{\textup{loop}}$ and enjoys the following properties. 
\begin{proposition}[{\cite[Proposition 6]{lawler}}]
    The Brownian loop measure satisfies the restriction property and is conformally invariant, i.e.
    \begin{enumerate}
        \item If $X'\subset X$, then $\d\mu_{X'}(\eta)=\mathbbm{1}_{\eta\subset X'}\d\mu_X(\eta)$ where $\mathbbm 1_{\eta\subset X'}$ is the indicator function;
        \item If $(X_1, g_1)=(X,g)$ and $(X_2, g_2)=(X,e^{2\sigma}g)$ are two conformally equivalent Riemann surfaces, with $\sigma\in C^\infty(X,\R)$, then $\mu_{X_1}=\mu_{X_2}$.
    \end{enumerate}
\end{proposition}
Notice in particular that $X'$ can have a boundary. When this is the case, we consider Brownian motion with Dirichlet conditions, i.e., that is killed upon hitting the boundary of $X'$.
The conformal invariance property allows us to view the Brownian loop measure as a measure on Riemann surfaces, and not merely on Riemannian ones.


\subsection{Selberg zeta functions}\label{section Selberg zeta function}

Let $X$ be a complete hyperbolic surface, and $\delta$ be the exponent of convergence of the associated Poincaré series. We have the following convergence lemma.

\begin{lemma} \label{convergence zeta function}
    The product defining the Selberg zeta function $Z_X$ converges for any $s \in \C$ such that $\mathrm{Re}\, s>\delta$. Moreover, $Z_X$ admits a meromorphic continuation to the entire complex plane. It is an entire function if $X$ has no cusps.
\end{lemma}

\begin{proof}
    We only recall the proof of convergence when $\mathrm{Re}(s)>\delta$ here.
    Let $F$ be a fundamental domain for $X$, and $z\in F$ be a point in this fundamental domain. Let $\gamma\in \mathcal P_X$. Since $F$ is a fundamental domain, $\gamma$ can be covered by a geodesic $\tilde \gamma$ in the hyperbolic plane $\H$ going through a point in $F$, which we will denote as $q$. Let $h\in \Gamma$ be a primitive element of axis $\tilde \gamma$. By the triangle inequality,
\begin{align*}
    d(z,h\cdot z)&\leq d(z,q)+d(q,h\cdot q)+d(h\cdot q,h\cdot z)
    \leq \ell_\gamma +2a,
\end{align*}
where $a$ is the diameter of $F$. Consequently, for $\mathrm{Re}(s)>\delta$
\[ \sum_{\gamma\in \mathcal P_X} e^{-s\ell_\gamma}\leq e^{2as}\sum_{h\in \Gamma} e^{-s d(z,h\cdot z)}<\infty. \]
This implies the convergence of the product defining $Z_X(s)$.
\end{proof}



Meromorphic continuation of $Z_X(s)$ from $\mathrm{Re}(s)>\delta$ to the whole complex plane was proven by L. Guillopé (\cite{Guillopé1990}) in the infinite-area case. In the finite-area case, it is a consequence of the Selberg trace formula (\cite{Borth}).

\begin{proposition}
    Let $X$ be a compact hyperbolic surface without boundary. Then $Z_X$ is an entire function of order 2, whose zeros consist of the spectral zeros at points $s$ such that $s(1-s)\in\mathrm{Spec}(\Delta_X)$ and topological zeros at points $s=-k$, $k\in\N$.
\end{proposition}

\begin{proof}

Let $1<\mathrm{Re}\,s<a$. Let us apply the Selberg trace formula to the function
$$h(r)=\frac{1}{r^2-(s-\frac 12)^2}-\frac{1}{r^2-(a-\frac 12)^2}.$$
We obtain
    \begin{align*}
        & \frac{1}{2s-1}\frac{Z_X'}{Z_X}(s)-\frac{1}{2a-1}\frac{Z_X'}{Z_X}(a)\\
        & \qquad \qquad=\sum_{k=0}^\infty\left(\frac{1}{\lambda_k-s(1-s)}-\frac{1}{\lambda_k-a(1-a)}\right)
        +\chi(X)\sum_{k=0}^\infty\left(\frac{1}{s+k}-\frac{1}{a+k}\right)
    \end{align*}
where $0=\lambda_0<\lambda_1\leq\lambda_2\leq\cdots$ are the eigenvalues of the Laplacian $\Delta_X$.

The Weyl asymptotics imply that $\lambda_k\sim ck$ for a certain constant $c$. This implies, in particular, that the sum on the right hand side converges absolutely for any $s\in\C$, except at the poles where $s(1-s)=\lambda_k$ or $s=-k$. Upon integration, since the poles are simple and have positive integer residues, we obtain the proposition. 
\end{proof}

Although we will not go into the details here, it is worth noting that the general case of the Selberg trace formula can be obtained from the Selberg zeta function by applying the residue theorem (see \cite{fischer} for more details).

\section{The Brownian Loop Measure in Homotopy Classes}\label{blm}

Let $X$ be a complete geometrically finite hyperbolic surface. The definition of the Brownian loop measure can be generalized to Brownian motion with a killing rate $\kappa$. This has already been mentioned by Y. Le Jan in \cite{lejan} and M. Ang et al. in \cite{AngPark}. We define Brownian motion with a killing rate $\kappa\geq-1/4$ as the process generated by the operator $-\Delta_X-\kappa$. The associated kernel on $X$ is
$$p_X^\kappa(t,z,z')\coloneqq p_X(t,z,z')e^{-\kappa t},$$
where $p_X$ denotes the heat kernel on $X$. The construction of the Brownian loop measure, which we denote by $\mu^\kappa_X$, also holds in the case of Brownian motion with killing. 

The main result of this part is the computation of the mass of $\mu_X^\kappa$ in homotopy classes $\mathcal C_X(\gamma^m)$ for $\gamma\in \mathcal P_X$ and $m\in \N^*$. Notice that this is the case of primary interest, since every essential free homotopy class contains a unique oriented closed geodesic, and the Brownian loop measure in non-essential (i.e., homotopically trivial or homotopic to a cusp) classes is infinite. Indeed, short time asymptotics of the trace of the heat kernel imply that the integral over small times does not converge. 
\begin{lemma}\label{thm:each}
    Suppose $(X,g)$ is a hyperbolic surface without boundary. Then, for any $\gamma\in \mathcal P_X$, $m\in \N^*$, $\kappa\geq-\frac14$,
    \[ \mu^\kappa_X(\mathcal C_X(\gamma^m)) = \frac{1}{m}\frac{e^{ m\left(\frac12-\sqrt{\frac14+\kappa}\right)\ell_\gamma }}{ e^{m\ell_\gamma}-1 }. \]
\end{lemma}
Lemma \ref{thm:each} generalizes \cite[Lemma 1.1]{BLM} 
which considered the $\kappa=0$ case. 
The proof of Lemma \ref{thm:each} presented here is an adaptation of that of \cite[Lemma 3.2]{BLM}, where Brownian motion without killing was considered.

\begin{proof}[Proof of Lemma \ref{thm:each}]
{\bf 1.}
Let $\gamma$ be an oriented closed geodesic on $X$ of length $\ell_\gamma$ and iteration number $m$. Notice that for all $x\in X$ and $t>0$, the following decomposition holds:
\begin{equation*}
    \W_{x\to x}^t(X)(\mathcal{C}_X(\gamma))=\sum_{h\in[\gamma]}|\W_{z\to h\cdot z}^t(\H)|
\end{equation*}
where $[\gamma]$ is the conjugacy class associated with $\gamma$ in $\Gamma$, $z\in \H$ is a lift of $x$, and $|\W_{z\to h\cdot z}^t(\H)|$ is the total mass of $\W_{z\to h\cdot z}^t(\H)$.

Indeed, fixing $\epsilon_0>0$ small enough, applying Lemma \ref{lemma:disintegrate}  for $\epsilon>0$ to the set
\[\begin{split} 
    \left\{ W\in C([0,t];X) \mid \ W_0=x, \ W_t\in D_{\epsilon,X}(x), \ d_X^{\infty}(W,\mathcal C_X(\gamma))<\epsilon_0 \right\}, 
\end{split}\]
where $d^{\infty}_X(W,\mathcal C_X(\gamma))<\epsilon$ means that there exists $\eta\in C([0,t];X)$ such that $\eta_0=\eta_t=x$, $[\eta]\in \mathcal C_X(\gamma)$ and $d^{\infty}_X(W,\eta)<\epsilon$. We thus find
\[\begin{split}
    & \mathbb W_{x\to x}^t(X)(\mathcal C_X(\gamma)) \\
    & = \lim_{\epsilon\to 0+}\frac{ \mathbb W_x^t(X)( \{ W\in C([0,t];X) \mid W_0=x, \ W_t\in D_{\epsilon, X}(x), \ d^{\infty}_X(W,\mathcal C_X(\gamma))<\epsilon_0 \} ) }{\textup{vol}(D_{\epsilon,X}(x))}.
\end{split}  \]
Let $z$ be a lift of $x$ in $\H$. Since Brownian motion only depends on the local geometry of the surface, its paths can be lifted to the universal cover. 
Let $\tilde W\in C([0,t];\H)$, with $\tilde W_0 =z$, be the lift of $W\in C([0,t];X)$, with $W_0=x$. Then, for $\epsilon$ sufficiently small,
\[\begin{split} 
W_t\in D_{\epsilon,X}(x), \ d^{\infty}_X(W, \mathcal C_X(\gamma))<\epsilon_0 
\ \Leftrightarrow \ \exists h\in [\gamma], \ \tilde W_t\in D_{\epsilon,\H}(h\cdot z).
\end{split}\]
Consequently,
\begin{align*}
&\mathbb{W}^{t}_{x\to x}(X)(\mathcal{C}_X(\gamma))\\
&= \lim_{\varepsilon\to 0+}
\sum_{h\in[\gamma]}
\frac{
\mathbb{W}^{t}_{z}(\mathbb{H})\bigl(\{ \tilde W \in C([0,t];\mathbb{H})\mid \tilde W_0=z, \ \tilde W_t \in D_{\varepsilon,\mathbb{H}}(h\cdot z)\}\bigr)
}{
\operatorname{vol}(D_{\varepsilon,\mathbb{H}}(h\cdot z))
}\\
& = \sum_{h\in[\gamma]} \bigl|\mathbb{W}^{t}_{z\to h\cdot z}(\mathbb{H})\bigr|.
\end{align*}

{\bf 2.}
Furthermore, if $p_\H(t;z,w):(0,\infty)\times\H\times\H\to \R$ denotes the heat kernel on the hyperbolic plane, then by Lemma \ref{lemma:disintegrate} 
\begin{align*}
    |\mathbb{W}_{z\rightarrow w}^t(\mathbb{H})|&=p_\mathbb{H}(t, z,w)=\frac{\sqrt{2}}{(4 \pi t)^{3/2}} e^{-t/4} \int_{d(z,w)}^{\infty} \frac{r e^{-r^2/(4t)}}{\sqrt{\cosh r - \cosh d(z,w)}} \, \mathrm{d}r
\end{align*}
where $d(z,w)$ is the hyperbolic distance between $z$ and $w$, see \cite[\S 2.1]{Borth}. Essentially, $p_\H(t,z,w)$ is the probability that a Brownian motion starting at $z$ is at $w$ at time $t$.

{\bf 3.}
Using the definition of $\mu^\kappa_X$, introduced at the beginning of \S \ref{blm}, we can express the mass of $\mathcal C_X(\gamma)$ under $\mu^\kappa_X$ in terms of the heat kernel $p_{\H}$:
$$\mu^\kappa_X(\mathcal C_X(\gamma)) = \int _0^\infty \frac{e^{-\kappa t}}{t}\sum_{h\in[\gamma]}\int_F p_\H(t,z,h\cdot z)\,\mathrm{d}\rho(z)\mathrm{d}t.
$$
Here, $F$ is a fundamental domain of $X$, and $\mathrm{d}\rho(z)$ denotes the hyperbolic area measure on $\H$. We now proceed as in the proof of the Selberg trace formula (see \cite{bergeron,Borth, BLM}).

Fix $h_0\in[\gamma]$. There exists a unique primitive element $\alpha\in \Gamma$ such that $h_0=\alpha^m$ and $m>0$. The centralizer of $h_0$ in $\Gamma$ is given by $\langle \alpha \rangle$. In particular, $[\gamma] = \{h_1^{-1}h_0h_1 \mid [h_1]\in \langle \alpha \rangle\backslash \Gamma \}$. Hence
\[\begin{split} 
    \sum_{h\in[\gamma]}\int_F p_\H(t,z,h\cdot z)\,\mathrm{d}\rho(z) & =\sum_{[h_1]\in\langle\alpha\rangle\backslash\Gamma}\int_{F}p_\H(t,z, h_1^{-1}h_0h_1\cdot z)\,\mathrm{d}\rho(z) \\
    & = \sum_{[h_1]\in \langle \alpha\rangle\backslash \Gamma} \int_F p_{\H}(t, h_1\cdot z, h_0\cdot(h_1\cdot z)) \,\mathrm{d}\rho(z)\\
    & = \sum_{[h_1]\in \langle \alpha\rangle\backslash \Gamma}\int_{h_1\cdot F} p_{\H}(t,z,h_0\cdot z)\,\mathrm{d}\rho(z).
\end{split}\]
Here, we used the fact that both the heat kernel and the area measure are invariant under isometries.

Up to conjugation, we may assume $h_0$ is given by $h_0:z\mapsto e^{\ell_\gamma}z$. Any point in $\H$ can be pulled back by an element of the group $\langle \alpha \rangle$ to a unique point in the strip $F_\alpha \coloneqq \{x+iy \mid 1\leq y<e^{\ell_\gamma/m}\}$. It implies that
\begin{align*}
    & \sum_{h\in[\gamma]}\int_F p_\H(t,z,h\cdot z)\,\mathrm{d}\rho(z)
    =\int_{F_\alpha}p_\H(t,z,e^{\ell_\gamma}z)\,\mathrm{d}\rho(z)\\
    &\qquad \qquad =\frac{\sqrt{2} e^{-t/4}}{(4 \pi t)^{3/2}} \int_1^{e^{\ell_\gamma/m}} \int_{\mathbb{R}} \int_{d(z,e^{\ell_\gamma} z)}^{\infty}\frac{r e^{-r^2 / 4t} \, dr}{\sqrt{\cosh r - \cosh d(z,e^{\ell_\gamma} z)}} \,\frac{\mathrm dx \, \mathrm dy}{y^2},
\end{align*}
where $z=x+iy$, and the hyperbolic distance $d$ verifies
\[ \cosh d(z,e^{\ell_\gamma}z)=1 + 2\left(\sinh \frac{\ell_\gamma}{2}\right)^{\! 2} \left(1 + \frac{x^{2}}{y^{2}}\right). \]
We thus compute 
\[ \sum_{h\in [\gamma]}\int_F p_{\H}(t,z,h\cdot z)\,\mathrm d\rho(z) = - \frac{\sqrt{2}e^{-\frac{t}{4}}}{(4\pi t)^{\frac32}}\cdot\frac{2t\ell_\gamma}{m} \int_\R\int_{\cosh^{-1}(q(\theta))}^\infty \frac{\mathrm{d}(e^{-\frac{r^2}{4t}})}{\sqrt{\cosh{r} - q(\theta)}}\mathrm \,\mathrm d\theta, \]
where we denote $q(\theta)\coloneqq 1+2(\sinh(\frac{\ell_\gamma}{2}))^2(1+\theta^2)$, and $\cosh^{-1}$ is the inverse function of $\cosh: (1,+\infty)\to (\cosh(1),+\infty)$. 

{\bf 4. } To proceed with the computation, we introduce the Abel transform $\mathcal A$ and its inverse transform (see \cite[Lemma 3.10]{bergeron}):
\[ \mathcal Au(\rho)\coloneqq 2\int_{\rho}^\infty u(r)\frac{r \,\mathrm d r}{\sqrt{r^2-\rho^2}}, \ \ \mathcal A^{-1}v(r)\coloneqq -\frac{1}{\pi}\int_r^\infty \frac{\mathrm d v(\rho)}{\sqrt{\rho^2-r^2}}. \]
We then compute:
\[\begin{split} 
    & \int_\R\int_{\cosh^{-1}(q(\theta))}^\infty \frac{\mathrm{d}(e^{-\frac{r^2}{4t}})}{\sqrt{\cosh{r} - q(\theta)}} \,\mathrm d\theta \\
    & \qquad \qquad = \int_\R\int_{\sqrt{q(\theta)}}^\infty \frac{ \mathrm{d}( e^{ -\frac{ (\cosh^{-1}(\rho^2))^2 }{4t} } ) }{ \sqrt{ \rho^2 - q(\theta) } }\,\mathrm d\theta \\
    & \qquad \qquad = -\pi\int_\R \mathcal A^{-1}\left( e^{ -\frac{ (\cosh^{-1}(\bullet^2))^2 }{4t} } \right)\left(\sqrt{q(\theta)}\right) \,\mathrm d\theta \\
    &  \qquad \qquad = -\frac{\pi}{\sqrt{2}\sinh(\frac{\ell_\gamma}{2})}\int_{\sqrt{\cosh{\ell_\gamma}}}^\infty \mathcal A^{-1}\left( e^{ -\frac{ (\cosh^{-1}(\bullet^2))^2 }{4t} } \right)(r)\frac{ r\mathrm d r }{\sqrt{r^2-\cosh{\ell_\gamma}}} \\
    & \qquad \qquad = -\frac{\pi}{2\sqrt{2}\sinh(\frac{\ell_\gamma}{2})}\mathcal A \mathcal A^{-1}\left( e^{ -\frac{ (\cosh^{-1}(\bullet^2))^2 }{4t} } \right)\left(\sqrt{\cosh{\ell_\gamma}}\right) \\
    &  \qquad \qquad= -\frac{\pi}{2\sqrt{2}\sinh(\frac{\ell_\gamma}{2})} e^{-\frac{\ell_\gamma^2}{4t}}.
\end{split}\]
We obtain:
$$\sum_{h\in[\gamma]}\int_F p_\H(t, z,h\cdot z)\,\mathrm{d}\rho(z)=\frac{1}{4\sqrt{\pi t}}e^{-\frac{t}{4}}e^{-\frac{\ell_\gamma^2}{4t}}\frac{\ell_\gamma}{m\sinh \frac{\ell_\gamma}{2}}.$$
Hence, we conclude:
\begin{align*}
    \mu_X^\kappa(\mathcal{C}_X(\gamma)) &= \int_0^\infty \frac{e^{-\frac{t}{4}}}{\sqrt{4\pi t}}  \frac{\ell_\gamma \, e^{-\frac{\ell_\gamma^2}{4t}}}{2m \sinh\left(\frac{\ell_\gamma}{2}\right)}  e^{-\kappa t}  \,\frac{\mathrm{d}t}{t} = \frac{1}{m}\frac{e^{\left(\frac{1}{2}-\sqrt{\frac{1}{4}+\kappa}\right)\ell_\gamma}}{e^{\ell_\gamma}-1},
\end{align*}
where the integral is computed using the formula
\[ t^{-\frac32}e^{-\frac{\alpha^2}{t}-\beta^2 t} =  \ -\frac{\pi}{2\alpha}\frac{\mathrm{d}}{\mathrm d t}\left( e^{2\alpha\beta} \mathrm{Erf}\left(\frac{\alpha+\beta t}{\sqrt t}\right) + e^{-2\alpha\beta} \mathrm{Erf}\left( \frac{\alpha -\beta t}{\sqrt t} \right) \right), \]
which holds for constants $\alpha$, $\beta>0$, and the error function $\mathrm{Erf}(w)\coloneqq \frac{2}{\sqrt \pi}\int_0^w e^{-t^2}\,\mathrm d t$.
\end{proof}

\Remarks
1. The formula in Lemma \ref{thm:each} remains valid for any complex parameter $\kappa$ such that $\mathrm{Re}\,\kappa \geq -1/4$. However, Brownian motion with a complex killing rate lacks a probabilistic interpretation. While many of the subsequent results continue to hold from the perspective of measure theory, we shall not pursue this generalization here.

\noindent 2. Notice that the Brownian loop measure with killing is not necessarily conformally invariant. In fact, this is the case if and only if $\kappa\neq0$, i.e., without killing.

Theorem \ref{thm:zeta} is a direct consequence of Lemma \ref{thm:each}.
\begin{proof}[Proof of Theorem \ref{thm:zeta}]
For $\text{Re}\,s>\delta$, using the Taylor expansion of the logarithm, one finds
$$-\log Z_X(s)=\sum_{\gamma\in \mathcal{P}_X}\sum_{m=1}^\infty \frac{1}{m}\frac{1}{e^{sm\ell_\gamma}-e^{(s-1)m\ell_\gamma}}.$$
In particular, for $\kappa>0$ and $s=\frac12+\sqrt{\frac14+\kappa}$, we have 
\begin{align*}
    -\log Z_X\left(\frac12+\sqrt{\frac14+\kappa}\right) 
    &= \sum_{\gamma \in \mathcal{P}_X} \sum_{m=1}^\infty \frac{1}{m} \cdot 
    \frac{1}{e^{\left(\frac{1}{2} + \sqrt{\frac{1}{4} +\kappa} \right) m \ell_\gamma} 
    - e^{-\left(\frac{1}{2} - \sqrt{\frac{1}{4} + \kappa} \right) m \ell_\gamma}} \\
    &= \sum_{\gamma \in \mathcal{P}_X} \sum_{m=1}^\infty 
    \frac{1}{m} \cdot \frac{e^{ \left(\frac{1}{2} - \sqrt{\frac{1}{4} + \kappa} \right) m \ell_\gamma}}{e^{m \ell_\gamma} - 1}.
\end{align*}
The calculations above work formally for any real killing rate on any complete hyperbolic surface without boundary. The formula therefore holds as long as convergence is guaranteed. 
Lemma \ref{convergence zeta function} shows this is the case as long as $s=\frac12+\sqrt{\frac14+\kappa}>\delta$.
\end{proof}

\section{Trace and determinant}\label{results}

\subsection{Brownian loop measure and regularized heat trace}\label{subsection:trace_reg}

Let $X=\Gamma\backslash \H$ be a complete hyperbolic surface. Let $\Delta$ be the (positive) Laplace-Beltrami operator on $X$.
The heat operator $e^{-t\Delta}$ is not necessarily of trace class. 
We thus introduce the following renormalization of the heat kernel.

\begin{definition}\label{definition:trace}
    Suppose $X$ is as above. Let $(c_1,\dots,c_{n_{c}})$ denote the cusps of $X$, and $[c_l]$ the homotopy class of the $l$'th cusp. We define $E_{X,\kappa}$ to be the operator whose Schwartz kernel is the heat kernel on $X$ with killing rate $\kappa \geq -1/4$, from which we subtracted the heat kernel on the plane and the parabolic terms. That is,
\begin{align*}
    E_{X,\kappa}(t, z,z')& \coloneqq \sum_{h\in\Gamma\backslash\{\mathrm{Id}\}} p_\H^\kappa(t, z,h\cdot z' )-\sum_{l=1}^{n_c}\sum_{h\in[c_l]}p_\H^\kappa (t,z,h\cdot  z')
\end{align*}
for $(z,z')\in X\times X$, which are identified with their lifts in $F\times F$, where $F$ is a fixed fundamental domain of $X$.
\end{definition}

\Remark
In the case where the surface does not have cusps, $E_{X,\kappa}$ is the heat kernel from which we subtracted the heat kernel on the entire hyperbolic plane. If $X$ has infinite area, this term is responsible for the divergence of the trace. It is, therefore, a natural regularization. A similar operation has been used by D. Borthwick in \cite[\S 3.3]{Borth} to regularize the resolvent operator. 

By definition, we have the following identity:
\begin{align*}
    E_{X,\kappa}(t,z,z')
    &=\sum_{\gamma\in\mathcal{P}_X} \sum_{h \in[\gamma]} \sum_{m=1}^\infty p_\H^\kappa (t,z,h^m \cdot z'), \ (z,z')\in F\times F.
\end{align*}

One can compute a Selberg-type trace formula for $E_{X,\kappa}(t)$.
\begin{proposition}\label{proposition:trace}
    Let $\kappa \geq -1/4$. For any $t>0$, the operator $E_{X,\kappa}(t)$ is trace class, and the following trace formula holds:
    \begin{align*}
        \mathrm{Tr}\,E_{X,\kappa}(t)&=\frac{e^{-t/4-\kappa t}}{4\sqrt{\pi t}}
        \sum_{\gamma \in \mathcal{P}_X}\sum_{k=1}^\infty \ell_\gamma \frac{e^{-\frac{(k\ell_\gamma)^2}{4t}}}{\sinh (\frac{k\ell_\gamma}{2})},
    \end{align*}
where $\mathcal{P}_X$ is the set of closed oriented primitive geodesics of $X$. 
\end{proposition}

\Remark
An alternative way to regularize the heat trace is to use the {\em zero heat trace} $\ztr(e^{-t\Delta})$; see \cite[\S 13.3]{Borth}. The heat trace introduced in the present paper is linked to the zero heat trace through the following identity, according to \cite[Lemma 13.13]{Borth}
\[ \ztr(e^{-t(\Delta+\kappa)}) = \mathrm{Tr}E_{X,\kappa}(t) + \textup{0-vol}(X)h_{0,\kappa} + n_c h_{c,\kappa}(t) \]
where $\textup{0-vol}(X)$ is the {\em zero volume} of $X$ and $h_{0,\kappa}$, $h_{c,\kappa}$ are functions defined by 
\[\begin{split} 
h_{0,\kappa}(t) & \coloneqq \frac{e^{-(\frac{1}{4}+\kappa) t}}{(4\pi t)^{\frac32}}\int_0^\infty \frac{ r e^{-\frac{r^2}{4t}} }{\sinh(\frac{r}{2})}\,\mathrm d r, \\ 
h_{c,\kappa}(t) & \coloneqq \left(-\frac{1}{2\pi}\int_{\R}\frac{ \Gamma'(1+iu) }{\Gamma(1+iu)} e^{-t u^2}\,\mathrm d u - \frac{ \log 2 }{\sqrt{4\pi t}} + \frac14\right) e^{-(\frac{1}{4}+\kappa)t}. 
\end{split}\]

\begin{proof}[Proof of Proposition \ref{proposition:trace}]
    By the monotone convergence theorem, Tr$\,E_{X,\kappa}(t)$ is obtained by integrating its kernel on the diagonal:
\begin{equation*}
    \text{Tr}\,E_{X,\kappa}(t)=\int_{X}E_{X,\kappa}(t,z,z)\,\mathrm d\rho(z),
\end{equation*}
where $d\rho$ is the hyperbolic volume form on the surface $X$. One obtains the desired result upon regrouping the elements of $\Gamma$ by conjugacy class, as is classically done in the proof of the Selberg trace formula (see \cite{bergeron} or \cite{Guillopé1986} for example).
\end{proof}

The operators $E_{X,\kappa}(t)$ are linked to the Brownian loop measure through the following proposition.

\begin{proposition}\label{link-trace-blm}
    Let $X$ be a complete and geometrically finite hyperbolic surface without boundary. Let $\delta$ be its exponent of convergence. For any $\kappa$ such that $\frac12+\sqrt{\frac14+\kappa}>\delta$, the following equality holds:
    \begin{equation*}
        \int_0^\infty \frac{1}{t}\operatorname{Tr}E_{X,\kappa}(t)\,\mathrm{d}t=\sum_{\gamma\in\mathcal{P}_X}\sum_{m=1}^\infty \mu_X^\kappa\big(\mathcal{C}_X(\gamma^m)\big).
    \end{equation*}
\end{proposition}

\begin{proof}
Let $(c_1,\dots,c_{n_c})$ denote the cusps of $X$, and $[c_l]$ the homotopy class of the $l$'th cusp. For any $\kappa$ such that $\frac12+\sqrt{\frac14+\kappa}>\delta$, by part {\bf 3} of the proof of Theorem \ref{thm:zeta},
\begin{align*}
            & \sum_{\gamma\in\mathcal{P}_X}\sum_{m=1}^\infty \mu_X^\kappa\big(\mathcal{C}_X(\gamma^m)\big)\\
            &\qquad =\sum_{\gamma\in\mathcal{P}_X}\sum_{m=1}^\infty\int _0^\infty \frac{1}{t}\sum_{h\in[\gamma]}\int_F p_\H^\kappa(t,z,h\cdot z)\,\mathrm{d}\rho(z)\mathrm{d}t.\\
            & \qquad=\int_0^\infty\frac{1}{t}\int_F \left(\sum_{h\in\Gamma\backslash\{\mathrm{Id}\}} p_\H^\kappa(t,z,h\cdot z)-\sum_{l=1}^{n_c}\sum_{h\in[c_l]}p_\H(t,z,h \cdot z)\right)\,\mathrm{d}\rho(z)\d t\\
            &\qquad =\int_0^\infty \frac{1}{t}\operatorname{Tr}E_{X,\kappa}(t)\,\mathrm{d}t.
        \end{align*}
The condition on $\kappa$ ensures that all the considered quantities are finite.
\end{proof}

\subsection{Brownian loop measure and regularized determinant of the Laplacian}

\subsubsection{The Compact Case}\label{compact case of asymptotic expansion}

In this section, we assume $X$ is a compact hyperbolic surface without boundary. The spectrum of the Laplace--Beltrami operator $\Delta$ on $X$ is an increasing discrete sequence
\[ 0 = \lambda_0 < \lambda_1 \le \lambda_2 \le \cdots .\]
The zero eigenvalue $\lambda_0$ has multiplicity one, and the corresponding eigenfunctions are the constant functions. Unfortunately, the quantity ``$\det \Delta =\prod_{k=1}^\infty \log \lambda_k$'' does not converge.
In order to overcome this difficulty in defining the determinant of the Laplacian, we use zeta-regularization, which is defined by taking
\[ \log\det\nolimits_{\zeta}\Delta \coloneqq -\zeta_X'(0) \]
where for $\operatorname{Re}(s) > 1$,
\[ \zeta_X(s) \coloneqq \frac{1}{\Gamma(s)} \int_0^\infty t^{\,s-1} 
\sum_{k=1}^\infty e^{-t\lambda_k}\, \mathrm dt, \]
which has a meromorphic extension to $\mathbb{C}$ and is analytic at $s=0$.

We now prove of item~\ref{item:compact_expansion} of Theorem~\ref{thm:det}.
We remark that according to Theorem~\ref{thm:zeta}, the left-hand side of the asymptotic formula in \ref{item:compact_expansion} of Theorem~\ref{thm:det} 
is also $-\log Z_X\left(\frac12+\sqrt{\frac14 +\kappa}\right)$. Consequently, 
\ref{item:compact_expansion} of Theorem~\ref{thm:det} can be considered as a corollary of the following identity
\[ {\det}_{\zeta}\left( \Delta_X+s(s-1) \right) = Z_X(s)\left( e^{E-s(s-1)}(2\pi)^s \frac{\Gamma_2(s)^2}{\Gamma(s)} \right)^{-\chi(X)}, \]
originally proved by Sarnak \cite[Theorem 1]{sarnak}. Here we give a direct proof for the asymptotic expansion.

\begin{proof}[Proof of \ref{item:compact_expansion} in Theorem \ref{thm:det}]
We first recall the following identity, which appears for instance in \cite[(3)]{naud} by F. Naud and is a consequence of the Selberg trace formula,
\begin{equation*}
    -\log\mathrm{det}_\zeta\Delta=-\mathrm{Area}(X)E-\gamma_{\textup{EM}}+\int_0^1\frac{S_X(t)}{t}\,\mathrm{d}t+\int_1^\infty\frac{S_X(t)-1}{t}\,\mathrm{d}t
\end{equation*}
where $S_X(t)$ is given by the formula
$$S_X(t)=\sum_{\gamma \in \mathcal{P}_X}\sum_{m=1}^\infty \frac{e^{-t/4}}{(4\pi t)^{1/2}}\frac{\ell_\gamma}{2\sinh(m\ell_\gamma/2)}e^{-\frac{(m\ell_\gamma)^2}{4t}}.$$
We recognize that $S_X(t)=e^{-\kappa t}\mathrm{Tr}E_{X,\kappa}(t)$ where $\mathrm{Tr}E_{X,\kappa}(t)$ is the renormalized heat trace introduced in Proposition \ref{link-trace-blm}.
In our context, it is important to note that $|S_X(t)-1|$ is exponentially small as $t\to +\infty$ and $S_X(t)$ is exponentially small as $t\to 0^+$, see \cite[\S 2]{naud}.

{\bf 1.}
We start by noticing that, according to Proposition \ref{link-trace-blm},
\[ \int_0^\infty e^{-\kappa t}\frac{S_X(t)}{t}\,\d t = \sum_{\gamma \in\mathcal{P}_X}\sum_{m=1}^\infty\mu_X^\kappa(\mathcal{C}_X(\gamma^m)).
\]

{\bf 2.} As $\kappa\to 0+$, since \(|(1-e^{-\kappa t}) S_X(t)/t|\leq C\kappa |S_X(t)|\) and $|S_X(t)|$ is integrable on $[0,1]$, we have 
\[ \int_0^1(1-e^{-\kappa t})\frac{S_X(t)}{t}\,\mathrm{d}t=O(\kappa). \]

{\bf 3.} 
The formula \cite[8.212-1]{gradshteyn2007} gives, as $\kappa\to 0^+$,
$$\int_1^\infty \frac{e^{-\kappa t}}{t}\,\d t = \int_\kappa^\infty t^{-1}e^{-t}\,\mathrm d t=  -\gamma_{\textup{EM}} - \log \kappa + O(\kappa).$$

{\bf 4. } Lastly, notice $|1-e^{-\kappa t}|\leq \kappa t$ for $t\in (1,-\log\kappa)$, hence 
\[ \int_1^{-\log\kappa}(1-e^{-\kappa t})\frac{S_X(t)-1}{t}\,\mathrm{d}t= O(\kappa(\log \kappa)^2). \]
On the other hand, since $|S_X(t)-1|$ decays exponentially as $t\to +\infty$, we may assume $|S_X(t)-1|=O(e^{-\beta t})$ for certain $\beta>0$. Then 
as $\kappa\to 0+$
\[ \int_{-\log\kappa}^\infty(1-e^{-\kappa t})\frac{S_X(t)-1}{t}\,\mathrm{d}t= O(e^{\beta\log \kappa}). \]

Collecting all integrals in {\bf 1}-{\bf 4} yields the desired asymptotics.
\end{proof}

For comparison, we mention that M. Ang et al. obtain a similar result in  \cite{AngPark}, for general closed Riemannian surfaces.
\begin{proposition}[{\cite[Proposition 1.4]{AngPark}}]\label{ang park asymptotic}
Let $(M,g)$ be a closed, smooth Riemannian surface. The mass of loops on $M$ under $\mu_M^\kappa$ with quadratic variation larger than $4\varpi>0$ is equal to:
\[
\frac{\mathrm{vol}_g(M)}{4\pi \varpi}
 - \frac{\chi(M)}{6}
 \left(\log \varpi + \gamma_{\textup{EM}}\right)
 - \log \kappa
 - \log {\det}_{\zeta} \Delta_g
 + \epsilon(\kappa,\varpi),
\]
where $\epsilon(\kappa,\varpi)$ tends to zero in the limit as first $\kappa$ and then $\varpi$ are sent to zero.
\end{proposition}
Although it is less general, as it only applies to compact hyperbolic surfaces, the asymptotic formula in \ref{item:compact_expansion} of Theorem \ref{thm:det} 
has the advantage of depending on only the killing rate $\kappa$, as it does not require cutting off loops with a small quadratic variation. 

An inspection of the proof of Lemma \ref{thm:each} shows that the total mass (with killing) of homotopically essential loops with quadratic variation smaller than $4\varpi$ is equal to 
\[ \int_0^{\varpi} \mathrm{Tr}E_{X,\kappa}(t)\,\frac{\mathrm d t}{t} \]
where $E_{X,\kappa}(t)$ is the renormalized heat operator as in Definition \ref{definition:trace}. 
Using Proposition~\ref{proposition:trace} and the fact that $S_X(t)$ decays exponentially as $t\to 0^+$, we conclude that this mass is of order $O(\varpi)$. 
By comparing \ref{item:compact_expansion} of Theorem \ref{thm:det} and 
Proposition \ref{ang park asymptotic}, we deduce the following asymptotic expansion for homotopically trivial loops with a quadratic variation greater than $4\varpi$:
\[\begin{split}
    \mu_X(\text{homotopically trivial loops with} & \text{ quadratic variation} \geq 4\varpi)\\
    & =\frac{\mathrm{Area}(X)}{4\pi}\left(\frac{1}{\varpi}+\frac{\log \varpi}{3}\right)+C+\epsilon(\varpi),
\end{split}\]
where $C\coloneqq-\mathrm{Area}(X)E+\left(\frac{\mathrm{Area}(X)}{12\pi} - 2\right)\gamma_{\mathrm{EM}}$ 
is a constant 
and $\epsilon(\varpi)\to 0$ as $\varpi\to 0$.

\subsubsection{The infinite-area case}\label{section:infinite}

When $X$ has an infinite area, the spectrum of the Laplacian is no longer discrete. In this case, the renormalization of the determinant of the Laplacian cannot be obtained directly from the Selberg trace formula. An alternative approach to renormalizing the determinant of the Laplacian,  using the 0-trace of the resolvent, was introduced by D. Borthwick, C. Judge, and P. A. Perry \cite{borth_judge_perry}. We briefly recall these renormalization methods here and develop the connections with the Brownian loop measure.

We first recall that, for a geometrically finite hyperbolic surface $X$ of infinite area, the resolvent of $\Delta_X$ on $X$ is defined for $\mathrm{Re}(s)>\frac12$, $s\notin [\frac12,1]$ as 
\[ R_X\coloneqq (\Delta_X + s(s-1))^{-1} : L^2(X)\to L^2(X). \]
It has been shown that the resolvent $R_X(s)$, when considered as a map from $C_0^{\infty}(X)$ to $C^{\infty}(X)$,
admits a meromorphic continuation to $\C$; see \cite[Chapter 6]{Borth} and references therein.
The poles of this meromorphic continuation of the resolvent are called the \textit{resonances} of $X$, and we use $\mathcal R_X$ to denote the set of resonances, counting multiplicities (see \cite[\S 8.1]{Borth} for the definition of multiplicities). 
The renormalized determinant of the Laplacian $\Delta_X$ is defined, according to Borthwick, Judge, and Perry \cite[\S 2]{borth_judge_perry}, through the differential equation
\[ \left( \frac{1}{2s-1}\frac{\mathrm d}{\mathrm d s} \right)^2\log{\det}_0(\Delta_X+s(s-1)) = \ztr(R_X(s)^2). \]
Here, ``$\ztr$'' denotes the 0-trace as in \cite[\S 10.1]{Borth}. We remark that $\det_0(\Delta_X+s(s-1))$ is defined up to two free parameters (by integrating the differential equation) and the choice of a boundary defining function of $X$ (by defining the 0-trace).

This renormalized determinant of the Laplacian happens to be well connected with the ``spectrum'' of $X$ (through resonances, to be more precise) and the length spectrum of $X$, as recalled in the proposition below. We first define a Weierstrass product
\[ P_X(s)\coloneqq s^{m(0)}\prod_{ \tau_j\in \mathcal R_X\setminus \{0\}  } \left( 1-\frac{s}{\tau_j} \right) e^{\frac{s}{\tau_j} + \frac{s^2}{2\tau_j^2} } = s^{m(0)}\exp\left( -\sum_{\tau_j\in \mathcal R_X\setminus \{0\}} \sum_{k\geq 3} \frac{1}{k}\frac{s^k}{\tau_j^k} \right) \]
where $m(0)$ is the resonance multiplicity of $0$. The interest in defining $P_X$ is that it converges uniformly on compact sets and thus defines an entire function for $s\in \C$. Moreover, the zero set of $P_X$ is exactly the set of resonances $\mathcal R_X$, counting multiplicities. We refer to \cite[\S 9.4]{Borth} and the references therein for details.

\begin{proposition}[{\cite{Borth}}, Theorems 10.1 and 10.14]\label{proposition:zeta_factor}
    Suppose $X$ is a non-elementary geometrically finite hyperbolic surface of infinite area. Then the Selberg zeta function $Z_X(s)$ admits the following factorization
    \begin{equation*}
        Z_X(s) = P_X(s)  e^{q(s)} G_{\infty}(s)^{-\chi(X)} \Gamma\left(s - \tfrac{1}{2}\right)^{n_c} ,
    \end{equation*}
    as well as
    \[ Z_X(s) = {\det}_0(\Delta_X+s(s-1)) e^{q_1(s)} G_{\infty}(s)^{-\chi(X)} \Gamma\left(s - \tfrac{1}{2}\right)^{n_c}(s-\tfrac12)^{\frac{n_c}{2}} \]
where $q$, $q_1$ are polynomials of degree at most $2$ and $G_{\infty}$, $\chi(X)$, $n_c$ are as in Theorem~\ref{thm:det}.
\end{proposition}


Theorem \ref{thm:zeta} induces a probabilistic interpretation of Proposition \ref{proposition:zeta_factor} by identifying $Z_X(\frac12+\sqrt{\frac14+\kappa})$, where $\kappa>- \frac 14$, with the mass of homotopically essential loops under the Brownian loop measure. Identities in Proposition \ref{proposition:zeta_factor} thus link the Brownian loop measure to the resonances of the surface $X$, as well as to the regularization of the Laplacian in the infinite-area case. 

\begin{proof}[Proof of Theorem \ref{thm:det}]
    The theorem is a straightforward consequence of previous propositions.
    Item~\ref{item:normlized_trace} of Theorem \ref{thm:det} is a reformulation of Proposition \ref{link-trace-blm}.
    Item~\ref{item:compact_expansion} is proved in \S \ref{compact case of asymptotic expansion}, 
    and item~\ref{item:noncompact_expansion} is a combination of the results of Proposition \ref{proposition:zeta_factor} and Theorem~\ref{thm:zeta}.
\end{proof}

\bibliographystyle{alpha}
\bibliography{loop.bib}

\end{document}